\documentclass[11pt]{amsart}
\usepackage[margin=3cm]{geometry}
\title{The two-sided peak polynomial}
\author{Alperen \"{O}zdemir} 
\address{Department of Mathematics, KTH Royal Institute of Technology, Stockholm, Sweden} 
\email{alpereno@kth.se} 
\thanks{This work was supported by the Swedish Research Council (VR), grant 2022-03875, and the Knut and Alice Wallenberg Foundation.}

\usepackage{amssymb, amsmath, amsthm, mathtools}
\usepackage[colorlinks=true, urlcolor=black,linkcolor=blue, citecolor=blue]{hyperref}
\usepackage{tikz,tikz-cd}
\usetikzlibrary{decorations.pathreplacing,calc}
\usepackage[shortlabels]{enumitem} 

\newtheorem{theorem}{Theorem}
\newtheorem{lemma}{Lemma}[section]

\theoremstyle{definition}

\newtheorem{remark}{Remark}[section]
\newtheorem{proposition}{Proposition}[section]

\ifdefined\NewCommandCopy
  \NewCommandCopy\latexbibitem\bibitem
\else
  \let\latexbibitem\bibitem
  \usepackage{xparse}
\fi

\ExplSyntaxOn
\RenewDocumentCommand{\bibitem}{O{}m}
 {
  \lullaby_bibitem:en { \lullaby_bibitem_check:nn { #1 } { \exp_not:n {#1}} } {#2}
 }
\cs_new_protected:Nn \lullaby_bibitem:nn { \latexbibitem[#1]{#2} }
\cs_generate_variant:Nn \lullaby_bibitem:nn { e }
\cs_new:Nn \lullaby_bibitem_check:nn
 {
  \str_case:nnF {#1}
   {
   {Mac15}{Mac1916}
    {Mac16}{Mac1916}
   }
   {#2}
 }
\ExplSyntaxOff

\begin{document}

\begin{abstract}
We derive a generating function identity for the joint distribution of the numbers of peaks of a permutation and its inverse, via enriched $P$-partitions. The coefficients of the corresponding peak polynomial $W_n(s,t)$ satisfy a second-order recurrence. A martingale formulation of this recurrence yields a bivariate central limit theorem, showing that the two statistics are asymptotically independent. We also give an exact closed form for their covariance, which is of order $n^{-1}$.
\end{abstract}

\maketitle

\section{Introduction}
A permutation $\pi \in S_n$ has a \emph{peak} at position $i$ if
$\pi(i-1) < \pi(i) > \pi(i+1)$, and we write $\mathrm{pk}(\pi)$ for
the number of peaks. The study of peaks originates with MacMahon
\cite{Mac15}, who enumerated permutations by descent and peak sets as
part of his systematic theory of permutation statistics. The
\emph{peak polynomial} is defined as
\[
W_n(t) = \sum_{\pi\in S_n} t^{\mathrm{pk}(\pi)+1}.
\]
Warren and Seneta \cite{WS96} proved that $W_n(t)$ is real-rooted. Stembridge's theory of enriched $P$-partitions \cite{Stem97} provides
the generating function identity
\begin{equation}\label{peakg0}
\frac{1}{2}\frac{(1+t)^{n+1}}{(1-t)^{n+1}}
W_n\!\left(\frac{4t}{(1+t)^2}\right) = \sum_{k\geq 0}(2k)^n\, t^k.
\end{equation}
This perspective was
further developed through the peak algebra of the symmetric group
\cite{Nym03} and its extension via type $B$ quasisymmetric functions
\cite{P06}.

Given this classical framework, we study the joint distribution of peaks in a permutation and its inverse. For a uniform $\pi \in S_n$, define the bivariate peak polynomial
\begin{equation}\label{Wndef}
W_n(s,t) = \sum_{\pi\in S_n} s^{\mathrm{pk}(\pi)+1}\,
t^{\mathrm{pk}(\pi^{-1})+1}.
\end{equation}

Our first result gives an explicit expression for $W_n(s,t)$ via enriched
$P$-partitions and Petersen's convolution identity \cite{P06}.

\begin{theorem}\label{genfnthm}
\[
\frac{1}{4}\left(\frac{(1+s)(1+t)}{(1-s)(1-t)}\right)^{n+1}
W_n\!\left(\frac{4s}{(1+s)^2},\frac{4t}{(1+t)^2}\right)
= \sum_{k,l\geq 0}\left(\sum_{m=1}^{n}\binom{2kl}{m}
\binom{n-1}{m-1}2^m\right) s^k t^l.
\]
\end{theorem}
The coefficients $w_n(a,b)$ of $W_n(s,t)$ satisfy a second-order
recurrence derived in Section~\ref{sec:genfn}. In contrast to the
univariate case, where real-rootedness of $W_n(t)$ implies stability,
$W_3(s,t)$ is not stable as a bivariate polynomial
(Proposition~\ref{prop:unstable}).

Warren and Seneta \cite{WS96} proved a central limit theorem for $\mathrm{pk}(\pi)$ for uniform permutations via local dependence. In \cite{Ozd22},
martingale methods are used to establish asymptotic normality along with Berry-Esseen bounds. A central limit theorem
for peaks in a fixed conjugacy class of $S_n$ was proved by Fulman,
Kim, and Lee \cite{FKL19} via analytic combinatorics. The joint distribution of descents in a permutation and its inverse
has been studied by Vatutin \cite{V96}, Chatterjee and Diaconis
\cite{Chat-Diaconis17}, and in \cite{Ozd22}; the present paper
studies peaks in a permutation and its inverse jointly.

 Let $M_n = \mathrm{pk}(\pi)+1$ and $M'_n = \mathrm{pk}(\pi^{-1})+1$
under the uniform distribution on $S_n$, with common mean $\mu_n =
(n+1)/3$ and variance $\sigma_n^2 = 2(n+1)/45$ from
\cite{WS96}. 
\begin{theorem}\label{mainthm0}
As $n\to\infty$,
\[
\left(\frac{M_n - \mu_n}{\sigma_n},\,
\frac{M'_n - \mu_n}{\sigma_n}\right) \;\Longrightarrow\; N(0, I_2).
\]
\end{theorem}

Our proof is based on a recurrence for the coefficients of $W_n(s,t)$,
which admits a martingale decomposition. The multivariate martingale
central limit theorem of Helland \cite{He82} then yields
convergence, with asymptotic independence following from the
conditional independence of the two coordinates at each step. An
exact formula for $\mathrm{Cov}(M_n,M_n')$
and its rate of decay are given in Section~\ref{sec:martingale}.

\section{The bivariate peak polynomial}\label{sec:genfn}

\subsection{Enriched $P$-partitions}
Stanley develops a theory of $P$-partitions in \cite{Stan72} (originally his 1971 Harvard PhD thesis), obtaining generating functions for descent statistics in rational form. Stembridge~\cite{Stem97} introduces enriched $P$-partitions and the 
algebra $\Pi$ of peak quasisymmetric functions, obtaining as a consequence 
the generating function \eqref{peakg0} for peaks of permutations. Consider the following linear ordering on $\mathbb{Z} \backslash \{0\}:$
\begin{equation*}
-1 < +1 < -2 < +2 < -3 < +3 < \cdots
\end{equation*}
Let $P$ be a finite partially ordered set with $n$ elements. An \textit{enriched $P$-partition} is a function $f: P \rightarrow \mathbb{Z} \backslash \{0\}$ such that for all $x \prec y$ in $P$, we have
\begin{itemize}
\item[(i)]$ f(x) \leq f(y),$
\item[(ii)]$ f(x)=f(y) > 0 $ implies $x \prec y,$
\item[(iii)]$ f(x)=f(y) < 0 $ implies $x \succ y.$
\end{itemize}

Define the \textit{order polynomial} $\Omega(P,k)$ to be the number of 
enriched $P$-partitions with range $[-k,k]$. By 
\cite[Corollary~1.1.1]{P06},
\begin{equation*}
\Omega(P,k)= \sum_{\pi \in \mathcal{L}(P)} \Omega(\pi,k)
\end{equation*}
where $\mathcal{L}(P)$ is the set of linear extensions of $P,$ identified with permutations in $S_n.$ An important special case is that if $P$ is an antichain over $n$ elements, then $\mathcal{L}(P)=S_n$ and $\Omega(P,k)=(2k)^n,$ which leads to \eqref{peakg0}. In particular, 
\begin{equation*}
\frac{1}{2} \frac{(1+t)^{n+1}}{(1-t)^{n+1}} \left( \frac{4t}{(1+t)^2} \right)^{\mathrm{pk}(\pi)+1}= \sum_{k \geq 0} \Omega(\pi, k) t^k.
\end{equation*}
The order polynomial can be evaluated for a given $\pi$ as
\begin{equation}\label{orderexp}
\Omega(\pi, k) = \sum_{m=1}^n c_m(\pi) \binom{k}{m},
\end{equation}
where $c_m(\pi)$ is the number of surjective enriched $\pi$-partitions taking values from $[-m,m]$ (integers between $-m$ and $m$). Gessel \cite[Theorem~11]{Ges84} 
proved a convolution identity for ordinary $P$-partitions; Petersen 
\cite{P06} extended it to the enriched setting:
\begin{equation}\label{multi}
\Omega(\pi,2kl) = \sum_{\tau \sigma = \pi} \Omega(\tau, k) \Omega(\sigma,l).
\end{equation}
\subsection{Proof of Theorem~\ref{genfnthm}}

Recall $W_n(s,t)$ from \eqref{Wndef}. Define
\begin{equation*}
F_{n}(s,t) \equiv \frac{1}{4} \left( \frac{(1+t)(1+s)}{(1-t)(1-s)}\right)^{n+1} W_n\!\left( \frac{4s}{(1+s)^2}, \frac{4t}{(1+t)^2} \right).
\end{equation*}
Applying the generating function identity \eqref{peakg0} in each variable
and summing over $\pi\in S_n$ gives
\[
F_n(s,t)=\sum_{k,l\ge 0}\!\left(\sum_{\pi\in S_n}\Omega(\pi,k)\,\Omega(\pi^{-1},l)\right)s^k t^l.
\]
Every factorisation $\tau\sigma=\mathrm{id}$ is of the form
$\tau=\pi$, $\sigma=\pi^{-1}$ for a unique $\pi\in S_n$, so
applying \eqref{multi} with $\pi=\mathrm{id}$ gives
\begin{equation*}
\sum_{\pi\in S_n}\Omega(\pi,k)\,\Omega(\pi^{-1},l)=\Omega(\mathrm{id},2kl).
\end{equation*}
Then we evaluate $c_m(\mathrm{id})$ by specialising $\pi=\mathrm{id}$
in \eqref{orderexp}. The identity permutation has no descents, so $f(i)=f(j)$ for $i<j$ only if $f(i)>0$.
Starting from $f(1)\in\{-1,+1\}$ gives two choices; at each subsequent step there are
three possibilities for $f(i+1)$, so
\[
c_m(\mathrm{id}) = 2\cdot[x^{m-1}](1+2x)^{n-1} = 2^m\binom{n-1}{m-1}.
\]
Substituting $k\mapsto 2kl$ in \eqref{orderexp} with $\pi=\mathrm{id}$ gives
\begin{equation}\label{Fexplicit}
F_n(s,t)= \sum_{k,l \geq 0} \left( \sum_{m=1}^{n} \binom{2kl}{m} \binom{n-1}{m-1} 2^m \right) s^k t^l,
\end{equation}
which is the statement of Theorem~\ref{genfnthm}. \qed

\subsection{The coefficient recurrence}

The coefficients $F_{n,kl}$ of $F_n(s,t)$, defined by \eqref{Fexplicit}, satisfy a
second-order linear recurrence. Setting
\begin{equation}\label{fnkdef}
f(n,k) = \sum_{m=1}^{n}\binom{k}{m}\binom{n-1}{m-1}2^m,
\end{equation}
so that $F_{n,kl} = f(n,2kl)$, introduce the ordinary generating function
\begin{equation}\label{Gkdef}
G_k(x) = \sum_{n=1}^{\infty} f(n,k)\,x^n.
\end{equation}
Swapping summation order and using
\[
\sum_{n=m}^{\infty}\binom{n-1}{m-1}x^n = \frac{x^m}{(1-x)^m}
\]
gives
\[
G_k(x) = \sum_{m=1}^{k}\binom{k}{m}\left(\frac{2x}{1-x}\right)^m
= \left(\frac{1+x}{1-x}\right)^k - 1.
\]
Setting
\[
Y = G_k(x)+1 = \left(\frac{1+x}{1-x}\right)^k,
\]
differentiation gives the ODE $(1-x^2)Y' = 2kY$. Extracting the coefficient of $x^{n-1}$
yields
\begin{equation}\label{frecurrence}
n\,f(n,k) = 2k\,f(n-1,k) + (n-2)\,f(n-2,k), \qquad n\ge 2,\quad f(1,k)=2k.
\end{equation}
Substituting $k\mapsto 2kl$ gives the recurrence for the coefficients of $F_n(s,t)$:
\begin{equation}\label{recurrencekl}
n\,F_{n,kl} = 4kl\,F_{n-1,kl} + (n-2)\,F_{n-2,kl}.
\end{equation}
Setting
\[
u = \frac{4s}{(1+s)^2}, \qquad v = \frac{4t}{(1+t)^2},
\]
summing \eqref{recurrencekl} over all $k,l$ and changing variables via
\begin{equation}\label{relations}
\frac{\partial}{\partial t}\!\left(\frac{1+t}{1-t}\right)^{\!n}
= \frac{2n(1+t)^{n-1}}{(1-t)^{n+1}},\quad
\frac{\partial}{\partial t}\frac{4t}{(1+t)^2}=\frac{4(1-t)}{(1+t)^3},\quad
1-\frac{4t}{(1+t)^2}=\frac{(1-t)^2}{(1+t)^2},
\end{equation}
and expressing the result in terms of $W_n(u,v)$ gives
\begin{align}\label{PDEWn}
nW_n =\;& n^2 uv\, W_{n-1}
+ 2n\!\left(uv(1-u)\,\partial_u W_{n-1}
           + uv(1-v)\,\partial_v W_{n-1}\right) \notag\\
&+ 4uv(1-u)(1-v)\,\partial_{uv}W_{n-1}
+ (n-2)(1-u)(1-v)\,W_{n-2}.
\end{align}
Extracting the coefficient of $u^a v^b$ on both sides of \eqref{PDEWn},
writing $W_n(u,v) = \sum_{a,b} w_n(a,b)\,u^a v^b$, gives the \emph{bivariate peak recurrence}:
\begin{align}
n\,w_n(a,b) &= (n-2a+2)(n-2b+2)\,w_{n-1}(a-1,b-1) \notag\\
&\quad + 4ab\,w_{n-1}(a,b) \notag\\
&\quad + 2a(n-2b+2)\,w_{n-1}(a,b-1) \label{eq:wnrec}\\
&\quad + 2b(n-2a+2)\,w_{n-1}(a-1,b) \notag\\
&\quad + (n-2)\bigl[w_{n-2}(a,b) - w_{n-2}(a-1,b)
         - w_{n-2}(a,b-1) + w_{n-2}(a-1,b-1)\bigr].\notag
\end{align}

\begin{remark}
The order of the recurrence \eqref{recurrencekl} reflects the dependence
between $\mathrm{pk}(\pi)$ and $\mathrm{pk}(\pi^{-1})$. If instead one
considers the product $W_n(s,1)\cdot W_n(1,t)$, corresponding to
independent peak statistics, its coefficients satisfy the first-order recurrence
\[
P(n,k) = (2k+2)\,P(n-1,k) + (n-2k)\,P(n-1,k-1),
\]
see \cite[Section~3]{WS96} or \cite[Remark~4.8]{Stem97}.
The second-order term $(n-2)F_{n-2,kl}$ in \eqref{recurrencekl}
thus reflects the dependence between a permutation and its
inverse. The bivariate peak recurrence
\eqref{eq:wnrec} is the starting point for the probabilistic analysis
of Section~\ref{sec:martingale}.
\end{remark}

\subsection{Instability of $W_n(s,t)$}

A polynomial $f(s,t)$ with nonnegative coefficients is \emph{stable} if
$f(s,t)\neq 0$ whenever $\mathrm{Im}(s)>0$ and $\mathrm{Im}(t)>0.$
The univariate polynomial $W_n(t) = W_n(t,1)$ is real-rooted
\cite{WS96}, hence stable in one variable. The bivariate polynomial $W_n(s,t)$ is
not stable.

\begin{proposition}\label{prop:unstable}
$W_3(s,t)$ is not stable.
\end{proposition}
\begin{proof}
If $f(s,t)$ is stable then $f(s,s)$ is real-rooted \cite[Lemma~2.4]{W11}.
From the definition,
\[
W_3(s,s) = s^2(s^2+2s+3).
\]
The factor $s^2+2s+3$ has discriminant $4-12=-8<0$, so $W_3(s,s)$
is not real-rooted, hence $W_3(s,t)$ is not stable.
\end{proof}

\begin{remark}
The product $W_n(s,1)\cdot W_n(1,t)$ is stable by
\cite[§10.13]{S00}, as a product of two real-rooted
univariate polynomials. That $W_n(s,t)$ itself fails
to be stable is therefore a consequence of the dependence between
$\mathrm{pk}(\pi)$ and $\mathrm{pk}(\pi^{-1})$, encoded in the
second-order term of \eqref{recurrencekl}.
\end{remark}

\section{The martingale decomposition}\label{sec:martingale}

Throughout this section $n\ge 3$ is fixed and $i=1,\ldots,n$ is the step
index. Let $\nu_n(a,b):=w_n(a,b)/n!$ denote the joint law of
$(M_n,M'_n)$, and write $\mathcal{F}_i=\sigma\bigl((M_1,M'_1),\dots,(M_i,M'_i)\bigr)$
and $\mathbf{E}_{i-1}[\,\cdot\,]=\mathbf{E}[\,\cdot\mid\mathcal{F}_{i-1}]$.
Define the scaled processes
\begin{equation}\label{Zdef}
  Z_i := i(i-1)\!\left(M_i-\frac{i+1}{3}\right), \qquad
  Z'_i := i(i-1)\!\left(M'_i-\frac{i+1}{3}\right),
\end{equation}
with $s_n:=n(n-1)\sigma_n$, so that at $i=n$ one has
$Z_n/s_n=(M_n-\mu_n)/\sigma_n$. The goal of this section is to establish the decomposition
\begin{equation}\label{Zdecomp}
  Z_n = \sum_{i=1}^n X_i + \sum_{i=1}^n i(i-1)\,Y_i,
  \qquad
  Z'_n = \sum_{i=1}^n X'_i + \sum_{i=1}^n i(i-1)\,Y'_i,
\end{equation}
where $\{X_i\}$ and $\{X'_i\}$ are martingale difference sequences with
respect to $\{\mathcal{F}_i\}$, and $Y_i, Y'_i \in \{-1,0,1\}$ are
second-order correction terms satisfying
$\mathbf{P}(Y_i\neq 0), \mathbf{P}(Y'_i\neq 0) = O(i^{-2})$.
The central limit theorem then follows by showing the correction sums are negligible
and the martingale-difference sums are jointly asymptotically normal.

\subsection{The principal term}\label{subsec:kernel}

Dividing \eqref{eq:wnrec} by $i\cdot i!$, the $w_{i-1}$-terms yield a
transition operator $\Lambda_i$ and the $w_{i-2}$-term yields a correction:
\begin{equation}\label{eq:precrec}
  \nu_i(a,b) = \Lambda_i\nu_{i-1}(a,b)
  + \frac{i-2}{i^2(i-1)}\,\Delta^{2}_i(a,b),
\end{equation}
where
\begin{multline}\label{eq:kernel}
  \Lambda_i\nu_{i-1}(a,b) := \frac{1}{i^2}\bigl[
    (i-2a+2)(i-2b+2)\,\nu_{i-1}(a-1,b-1)
    + 2b(i-2a+2)\,\nu_{i-1}(a-1,b) \\
    + 2a(i-2b+2)\,\nu_{i-1}(a,b-1)
    + 4ab\,\nu_{i-1}(a,b)\bigr],
\end{multline}
and
\begin{equation}\label{eq:Delta2}
  \Delta^{2}_i(a,b) :=
  \nu_{i-2}(a,b)-\nu_{i-2}(a-1,b)-\nu_{i-2}(a,b-1)+\nu_{i-2}(a-1,b-1).
\end{equation}
The weights in $\Lambda_i$ factor in the two coordinates: each term
in \eqref{eq:kernel} is a product of a factor depending only on $a$
and one depending only on $b$. Consequently, conditionally on
$\mathcal{F}_{i-1}$ the increments $U_i, U'_i\in\{0,1\}$ are
independent, each governed by
\begin{equation}\label{eq:Bmove}
  \mathbf{P}(U_i=1\mid\mathcal{F}_{i-1})=\frac{i-2M_{i-1}}{i},
  \qquad
  \mathbf{P}(U_i=0\mid\mathcal{F}_{i-1})=\frac{2M_{i-1}}{i},
\end{equation}
and likewise for $U'_i$. The first-order contributes the increment
\begin{equation}\label{Xdef}
  X_i :=
  \begin{cases}
    2(i-1)\!\left(M_{i-1}-\dfrac{i+1}{3}\right) - \dfrac{i(i-1)}{3}, & \text{prob.\ }
      \dfrac{2M_{i-1}}{i},\\[8pt]
    2(i-1)\!\left(M_{i-1}-\dfrac{i+1}{3}\right) + \dfrac{2i(i-1)}{3}, & \text{prob.\ }
      \dfrac{i-2M_{i-1}}{i},
  \end{cases}
\end{equation}
to $Z_i - Z_{i-1}$, satisfying $\mathbf{E}_{i-1}[X_i]=0$, and analogously
$X'_i$ for the primed coordinate. Hence $\{X_i\}$ and $\{X'_i\}$ are
martingale difference sequences with respect to $\{\mathcal{F}_i\}$.

\subsection{The second-order correction}\label{subsec:correction}

The transition operator $\Lambda_i$ alone does not fully reproduce $\nu_i$.
The remaining term in \eqref{eq:precrec},
\begin{equation}\label{eq:correctionterm}
  \frac{i-2}{i^2(i-1)}\,\Delta^{2}_i(a,b),
\end{equation}
is the source of dependence between $M_n$ and $M'_n$, and we will
show it is asymptotically negligible.

The correction term \eqref{eq:correctionterm} has the following interpretation as a mass transfer on $\mathbb{Z}^2$.
For any $(a,b)$, the term $\nu_{i-2}(a,b)$ appears
positively in $\Delta^2_i(a,b)$ and $\Delta^2_i(a+1,b+1)$, and
negatively in $\Delta^2_i(a+1,b)$ and $\Delta^2_i(a,b+1)$. This
suggests transferring mass $\nu_{i-2}(a,b)/2$ from $(a+1,b)$ and
$(a,b+1)$ toward $(a,b)$ and $(a+1,b+1)$. Applying this to every
$(a,b)$ simultaneously and reading the result as transition
probabilities from the state after applying $\Lambda_i$, we define
$(Y_i,Y'_i)\in\{-1,0,1\}^2$ by
\begin{equation}\label{Ydef}
(Y_{i},Y_{i}') =
\begin{cases}
(-1,\;0), & \text{prob.\ }
  \dfrac{i-2}{2i^2(i-1)}\,\nu_{i-2}(a-1,b), \\[8pt]
(+1,\;0),  & \text{prob.\ }
  \dfrac{i-2}{2i^2(i-1)}\,\nu_{i-2}(a,b-1), \\[8pt]
(0,\;-1), & \text{prob.\ }
  \dfrac{i-2}{2i^2(i-1)}\,\nu_{i-2}(a,b-1), \\[8pt]
(0,\;+1),  & \text{prob.\ }
  \dfrac{i-2}{2i^2(i-1)}\,\nu_{i-2}(a-1,b), \\[8pt]
(0,\;0),  & \text{otherwise,}
\end{cases}
\end{equation}
where $(a,b)$ denotes the value of $(M_i,M'_i)$ after applying $\Lambda_i$.

\begin{lemma}\label{lem:correction}
For every $i\ge 3$,
\[
\mathbf{P}(Y_i\neq 0)\le\frac{2(i-2)}{i^2(i-1)}\le\frac{2}{i^2},
\]
and the same bound holds for $\mathbf{P}(Y'_i\neq 0)$.
\end{lemma}
\begin{proof}
The bound follows from summing the non-zero probabilities in
\eqref{Ydef} and using the fact that $\nu_{i-2}$ is a probability measure.
\end{proof}

\begin{proposition}\label{prop:law}
For every $i=1,\ldots,n$, the pair $(M_i,M'_i)$ has joint law $\nu_i$.
In particular $M_n=\mathrm{pk}(\pi)+1$ and
$M'_n=\mathrm{pk}(\pi^{-1})+1$ for $\pi$ uniform on $S_n$.
\end{proposition}

\begin{proof}
We proceed by induction on $i$, the cases $i=1,2$ being immediate. Assume $(M_{i-1},M'_{i-1})$
has law $\nu_{i-1}$ and $(M_{i-2},M'_{i-2})$ has law $\nu_{i-2}$.

We compute $\mathbf{P}(M_i=a, M'_i=b)$ in two parts. For the first-order term, the only states that can transition to $(a,b)$ under $\Lambda_i$
are $(a-1,b-1)$, $(a-1,b)$, $(a,b-1)$, and $(a,b)$ itself. Weighting
each by its transition probability from \eqref{eq:kernel} and by the
inductive hypothesis that $(M_{i-1},M'_{i-1})$ has law $\nu_{i-1}$,
the total mass assigned to $(a,b)$ after the first-order step is
$\Lambda_i\nu_{i-1}(a,b)$. The correction step then adds \eqref{eq:correctionterm} to $(a,b)$, as verified by direct inspection of \eqref{eq:Delta2} and \eqref{Ydef}. Applying
\eqref{eq:precrec} gives
\[
\mathbf{P}(M_i=a,M'_i=b)
= \Lambda_i\nu_{i-1}(a,b) + \frac{i-2}{i^2(i-1)}\Delta^2_i(a,b)
= \nu_i(a,b). \qedhere
\]
\end{proof}

The exact one-step relation is therefore
\begin{equation}\label{onestep-exact}
M_{i} = M_{i-1} + U_i + Y_i,\qquad M'_{i} = M'_{i-1} + U'_i + Y'_i,
\end{equation}
and telescoping from $i=1$ to $n$ yields \eqref{Zdecomp}.

\subsection{The joint central limit theorem}

Recall the decomposition \eqref{Zdecomp}, where $\{X_i\}$ and $\{X'_i\}$
are martingale difference sequences and $Y_i$, $Y'_i$ are second-order
corrections. We show in Steps 1-2 below that the corrections are
negligible and the martingale part is asymptotically normal; Slutsky's theorem
then gives the result.

The Gaussian limit follows from the multivariate martingale central limit
theorem of Helland \cite[Thm.~3.3]{He82}, which we state in the
form needed here.

\begin{theorem}[\cite{He82}]\label{thm:helland}
Let $\{(\xi_{n,i},\xi'_{n,i})\}_{i\le n}$ be a square-integrable
martingale-difference array with respect to $\{\mathcal{F}_i\}$. Suppose,
as $n\to\infty$,
\begin{align}
\sum_{i=1}^n\mathbf{E}_{i-1}[\xi_{n,i}^2]\xrightarrow{p}1,
\qquad
\sum_{i=1}^n\mathbf{E}_{i-1}[\xi_{n,i}'^2]\xrightarrow{p}1, \label{Hvar}\\
\sum_{i=1}^n\mathbf{E}_{i-1}[\xi_{n,i}\,\xi'_{n,i}]\xrightarrow{p}0, \label{Hcross}\\
\sum_{i=1}^n\mathbf{E}_{i-1}\!\bigl[(\xi_{n,i}^2+\xi_{n,i}'^2)\,
\mathbf{1}_{\{|\xi_{n,i}|+|\xi'_{n,i}|>\varepsilon\}}\bigr]\xrightarrow{p}0
\quad\text{for all }\varepsilon>0. \label{Hlind}
\end{align}
Then $\bigl(\sum_i\xi_{n,i},\sum_i\xi'_{n,i}\bigr)\Rightarrow N(0,I_2)$.
\end{theorem}

\subsubsection*{Step 1: The corrections are negligible}

\begin{lemma}\label{lem:negligible}
$\frac{1}{s_n}\sum_{i=1}^n i(i-1)\,Y_i\to0$ almost surely,
and likewise with $Y'_i$.
\end{lemma}

\begin{proof}
By Lemma~\ref{lem:correction}, $\sum_{i=1}^\infty\mathbf{P}(Y_i\neq 0)
\le\sum_{i=1}^\infty\frac{2}{i^2}<\infty$. The first Borel--Cantelli
lemma gives $Y_i=0$ for all $i$ sufficiently large, almost surely.
Hence $\sum_{i=1}^n i(i-1)Y_i$ is eventually a finite random constant,
and dividing by $s_n\to\infty$ gives the claim. The same argument
applies to $Y'_i$.
\end{proof}

\begin{remark}
The corrections $Y_i$ and $Y'_i$ are the only source of dependence
between $M_n$ and $M'_n$. Their rarity $\mathbf{P}(Y_i\neq0)=O(i^{-2})$
makes Borel--Cantelli applicable regardless of any correlation structure,
giving the stronger almost sure conclusion.
\end{remark}

\subsubsection*{Step 2: The martingale part is jointly asymptotically normal}

\begin{lemma}\label{lem:cleanCLT}
$\displaystyle\Bigl(\frac1{s_n}\sum_{i=1}^n X_i,\ \frac1{s_n}\sum_{i=1}^n X'_i\Bigr)\Rightarrow N(0,I_2).$
\end{lemma}

\begin{proof}
Apply Theorem~\ref{thm:helland} to $\xi_{n,i}=X_i/s_n$ and
$\xi'_{n,i}=X'_i/s_n$. Write
\begin{equation*}
V_i := M_i - \frac{i+1}{3}
\end{equation*}
for the centered process.

\noindent \emph{Lindeberg condition \eqref{Hlind}.}
Both outcomes of $X_i$ in \eqref{Xdef} equal $2(i-1)V_{i-1}$
plus a term of size $\le\frac{2}{3}i(i-1)$. Since $\mathrm{pk}(\pi)\le\lfloor(i-1)/2\rfloor$
for $\pi\in S_{i-1}$, we have $1\le M_{i-1}\le\lfloor i/2\rfloor$,
so $|V_{i-1}|\le\frac{i}{3}$,
giving the deterministic bound $|X_i|\le C\,i^2$ uniformly. Hence
\[
\max_{1\le i\le n}\frac{|X_i|}{s_n}\le\frac{C\,n^2}{s_n}=O(n^{-1/2})\to0,
\]
and the same for $X'_i$, so the indicators in \eqref{Hlind} vanish
for all large $n$.

\noindent \emph{Conditional variance \eqref{Hvar}.}
By \eqref{Xdef}, $X_i$ takes two values
differing by $i(i-1)$, with
$\mathbf{P}(U_i=0\mid\mathcal{F}_{i-1})=2M_{i-1}/i$.
Writing $\frac{2M_{i-1}}{i}=\frac{2}{3}+\frac{2V_{i-1}}{i}$,
\[
\mathbf{E}_{i-1}[X_i^2]=\left(\frac{2}{9}
-\frac{2V_{i-1}}{3i}
-\frac{4V_{i-1}^2}{i^2}\right)i^2(i-1)^2.
\]
Summing over $i$ and dividing by $s_n^2$,
\[
\frac{1}{s_n^2}\sum_{i=1}^n\mathbf{E}_{i-1}[X_i^2]
=\frac{\frac{2}{9}\sum_i i^2(i-1)^2}{s_n^2}
-\frac{R_n^{(1)}}{s_n^2}-\frac{R_n^{(2)}}{s_n^2},
\]
where $R_n^{(1)}=\frac{2}{3}\sum_i i(i-1)^2V_{i-1}$
and $R_n^{(2)}=4\sum_i(i-1)^2V_{i-1}^2$.

\emph{Main term.} Since $\sigma_n^2=\frac{2(n+1)}{45}$ \cite{WS96} and
$s_n^2=n^2(n-1)^2\sigma_n^2$, we have $s_n^2\sim\frac{2}{45}n^5$.
Since $\sum_{i=1}^n i^2(i-1)^2\sim\frac{n^5}{5}$, the main term satisfies
\[
\frac{\frac{2}{9}\sum_i i^2(i-1)^2}{s_n^2}
\sim \frac{\frac{2}{9}\cdot\frac{n^5}{5}}{\frac{2}{45}n^5}
= \frac{2/45}{2/45} = 1.
\]

\emph{Term $R_n^{(1)}$.} Since $\mathbf{E}[V_{i-1}]=0$
we have $\mathbf{E}[R_n^{(1)}]=0$. Using
$\mathbf{E}[|V_{i-1}|]\le\sigma_{i-1}=O(\sqrt{i})$
\cite{WS96},
\[
\frac{\mathbf{E}[|R_n^{(1)}|]}{s_n^2}
\le\frac{\frac{2}{3}\sum_i i(i-1)^2\sigma_{i-1}}{s_n^2}
=O\!\left(\frac{n^{9/2}}{n^5}\right)=O(n^{-1/2})\to0,
\]
so $R_n^{(1)}/s_n^2\to0$ in $L^1$ and in probability.

\emph{Term $R_n^{(2)}$.} Since
$\mathbf{E}[V_{i-1}^2]=\sigma_{i-1}^2\sim 2i/45$ from \cite{WS96},
\[
\frac{\mathbf{E}[R_n^{(2)}]}{s_n^2}
=\frac{4\sum_i(i-1)^2\sigma_{i-1}^2}{s_n^2}
=O\!\left(\frac{n^4}{n^5}\right)=O(n^{-1})\to0,
\]
and $R_n^{(2)}/s_n^2\to0$ in probability by Markov's inequality.

Combining, $\sum_i\mathbf{E}_{i-1}[X_i^2]/s_n^2\xrightarrow{p}1$,
and likewise for the primed coordinate by symmetry.

\noindent \emph{Cross term \eqref{Hcross}.}
$X_i$ is a function of $U_i$ and $X'_i$ of $U'_i$; given
$\mathcal{F}_{i-1}$ these are conditionally independent by
\eqref{eq:Bmove}, and $\mathbf{E}_{i-1}[X_i]=\mathbf{E}_{i-1}[X'_i]=0$
by \eqref{Xdef}, so
\[
\mathbf{E}_{i-1}[X_iX'_i]=\mathbf{E}_{i-1}[X_i]\,
\mathbf{E}_{i-1}[X'_i]=0\qquad\text{for every }i.
\]
Thus $\sum_i\mathbf{E}_{i-1}[X_iX'_i]=0$, and \eqref{Hcross} holds exactly.

All hypotheses of Theorem~\ref{thm:helland} hold, giving the claim.
\end{proof}

\begin{proof}[Proof of Theorem~\ref{mainthm0}]
Dividing \eqref{Zdecomp} by $s_n$,
\[
\Bigl(\frac{Z_n}{s_n},\frac{Z'_n}{s_n}\Bigr)
=\underbrace{\Bigl(\frac1{s_n}\sum_i X_i,\ \frac1{s_n}\sum_i X'_i\Bigr)}_{\Rightarrow\,N(0,I_2)\ \text{(Lemma~\ref{lem:cleanCLT})}}
+\underbrace{\Bigl(\frac1{s_n}\sum_i i(i-1)Y_i,\ \frac1{s_n}\sum_i i(i-1)Y'_i\Bigr)}_{\xrightarrow{p}\,0\ \text{(Lemma~\ref{lem:negligible})}} .
\]
By Slutsky's theorem, $\bigl(Z_n/s_n,Z'_n/s_n\bigr)\Rightarrow N(0,I_2)$.
Since $Z_n/s_n = \dfrac{M_n-\mu_n}{\sigma_n}$ and
$Z'_n/s_n = \dfrac{M'_n-\mu_n}{\sigma_n}$, the theorem follows. Asymptotic independence follows from convergence to the standard bivariate normal with identity covariance matrix.
\end{proof}

\subsection{Exact covariance and rate of decorrelation}

The asymptotic independence in Theorem~\ref{mainthm0} is complemented by the following exact result, which identifies the rate at which $\mathrm{Cov}(M_n,M'_n)\to 0$.

\begin{proposition}\label{prop:indep}
For all $n \geq 2,$
\[
\mathrm{Cov}(M_n, M'_n) = \frac{n-2}{3n(n-1)}.
\]
\end{proposition}

\begin{proof}
Define $E_n = \mathbf{E}[M_n M'_n].$ Applying
$\frac{1}{n!}\partial_{uv}$ to the PDE \eqref{PDEWn} and evaluating
at $u=v=1$ via the product rule
$\partial_{uv}(fg)|_1 = f_{uv}g + f_u g_v + f_v g_u + f\,g_{uv}$
and $(1-u)|_1=(1-v)|_1=0,$ the five terms contribute
(with $P_{n-1} = \mathbf{E}[M_{n-1}] = n/3$):
\begin{align*}
\mathrm{T1}\ (n^2 uv\cdot W_{n-1})&:
    \quad n\bigl(1 + 2P_{n-1} + E_{n-1}\bigr),\\
\mathrm{T2}\ (2nuv(1-u)\cdot\partial_u W_{n-1})&:
    \quad -2\bigl(P_{n-1} + E_{n-1}\bigr),\\
\mathrm{T3}\ (2nuv(1-v)\cdot\partial_v W_{n-1})&:
    \quad -2\bigl(P_{n-1} + E_{n-1}\bigr),\\
\mathrm{T4}\ (4uv(1-u)(1-v)\cdot\partial_{uv} W_{n-1})&:
    \quad \tfrac{4}{n}E_{n-1},\\
\mathrm{T5}\ ((n-2)(1-u)(1-v)\cdot W_{n-2})&:
    \quad \tfrac{n-2}{n(n-1)}.
\end{align*}
Assembling gives
\begin{equation}\label{Erec}
nE_n = \left(n-4+\tfrac{4}{n}\right)E_{n-1}
+ 2(n-2)P_{n-1} + n + \tfrac{n-2}{n(n-1)}.
\end{equation}
Set $C_n = \mathrm{Cov}(M_n, M'_n) = E_n - \mu_n^2.$
Substituting $E_n = C_n + \mu_n^2$ and $E_{n-1} = C_{n-1} + \mu_{n-1}^2$
into \eqref{Erec}, using $\mu_n^2 = (n+1)^2/9,$ $\mu_{n-1}^2 = n^2/9,$
and $P_{n-1} = n/3,$ the deterministic terms on the right of \eqref{Erec} are
\begin{align*}
&\left(n-4+\tfrac{4}{n}\right)\frac{n^2}{9}
+ \frac{2n(n-2)}{3} + n + \frac{n-2}{n(n-1)}
- n\cdot\frac{(n+1)^2}{9}\\
&= \frac{-6n^2+3n}{9} + \frac{2n(n-2)}{3} + n + \frac{n-2}{n(n-1)}
= \frac{n-2}{n(n-1)},
\end{align*}
where the intermediate step uses $\frac{(n-4+4/n)n^2 - n(n+1)^2}{9} = \frac{-6n^2+3n}{9}.$
Hence
\begin{equation}\label{Crec}
C_n = \frac{(n-2)^2}{n^2}\,C_{n-1} + \frac{n-2}{n^2(n-1)}.
\end{equation}
The homogeneous solution of \eqref{Crec} is
\[
C_n^{(h)} = K\prod_{k=3}^{n}\frac{(k-2)^2}{k^2}
= K\cdot\frac{4}{n^2(n-1)^2},
\]
since $\prod_{k=3}^n\frac{k-2}{k} = \frac{2}{n(n-1)}$ by telescoping.
By variation of parameters, the general solution is
\[
C_n = \frac{4}{n^2(n-1)^2}\left[K + \frac{1}{4}\sum_{k=3}^{n}(k-2)(k-1)\right].
\]
The sum evaluates as
\[
\sum_{k=3}^{n}(k-2)(k-1) = \sum_{j=1}^{n-2}j(j+1) = \frac{(n-2)(n-1)n}{3}.
\]
Since $M_2 = M'_2 = 1$ deterministically, $C_2 = 0,$ giving $K=0.$ Therefore
\[
C_n = \frac{4}{n^2(n-1)^2}\cdot\frac{(n-2)(n-1)n}{12} = \frac{n-2}{3n(n-1)}.
\]
\end{proof}

\bibliographystyle{alpha}
\bibliography{peak}

\end{document}